\documentclass[12pt,reqno]{amsart}
\usepackage{amssymb}
\usepackage{amsmath}
\usepackage{amsfonts}
\usepackage{mathrsfs}
\usepackage[dvipsnames]{xcolor}
\definecolor{verydarkblue}{HTML}{000099}
\usepackage{color}
\usepackage{enumerate}
\usepackage[T1]{fontenc}
\usepackage[latin1]{inputenc}
\usepackage{ae,aecompl}
\usepackage{graphicx}
\usepackage{nicefrac}
\usepackage[includehead,includefoot,margin=19mm]{geometry}
\usepackage{bigstrut}
\usepackage{hyperref}
\hypersetup{
    colorlinks=true,       
    linkcolor=verydarkblue,          
    citecolor=verydarkblue,        
    filecolor=verydarkblue,      
    urlcolor=verydarkblue           
}

\usepackage{amssymb}
\usepackage{amsmath}
\usepackage{amsfonts}
\usepackage{mathrsfs}
\usepackage[all]{xy}
\usepackage{hyperref}  
\usepackage{tikz}
\usepackage{verbatim}
\usepackage{mathrsfs}

\theoremstyle{plain}
\newtheorem{theorem}{Theorem}[section]

\newtheorem{lemma}[theorem]{Lemma}

\newtheorem{corollary}[theorem]{Corollary}
\newtheorem*{theorem*}{Theorem}

\theoremstyle{remark}
\newtheorem{definition}[theorem]{Definition}

\newtheorem{example}[theorem]{Example}

\theoremstyle{remark}

\newtheorem{remark}[theorem]{Remark}

\def\ZZ{{\mathbb Z}}

\def\QQ{{\mathbb Q}}
\def\CC{{\mathbb C}}

\def\spec{{\mathrm{Spec}}}

\def\ker{{\mathrm{ker}}}

\def\deg{{\mathrm{deg}}}

\def\div{{\mathrm{div}}}

\def\G{{\mathbb{G}}}

\def\A{{\mathbb{A}}}

\def\P{{\mathbb{P}}}

\def\D{{\mathfrak{D}}}

\def\TT{{\mathbb{T}}}

\def\Dbar{{\mathfrak{D}_{\bar{k}}}}

\title[Horizontal $\mathbb{G}_{a}$-actions]
{Horizontal $\mathbb{G}_{a}$-actions on affine
\\ $\TT$-varieties  of complexity one}
\author{Kevin Langlois}
\date{}

\address{Mathematisches Institut, Heinrich Heine Universit\"at, 40225 D\"usseldorf, Germany.}
\email{langlois.kevin18@gmail.com}

\begin{document}

\thanks{\em 2010 Mathematics Subject Classification \rm  13A35, 14M25, 14R20. 
}

\maketitle

\begin{abstract}
We classify the $\mathbb{G}_{a}$-actions on normal affine varieties defined over any field that are horizontal with respect to a torus action of complexity one. 
This generalizes previous results that were available for perfect
ground fields (cf. \cite{FZ05, Lie10a, LL16}).
\end{abstract}

\section{Introduction}\label{sec-intro}

In this paper, we are considering algebraic varieties (that is, integral separated finite type schemes) over a field $k$.  We are interested in some classification problems for algebraic group actions. Namely, we study normal affine varieties endowed
with an action of the additive group $\mathbb{G}_{a}$ of the base field $k$. The further condition that we impose is that the $\mathbb{G}_{a}$-action has to be normalized by a torus action of complexity one. In the case where $k$ is algebraically closed of characteristic $0$, this classification was obtained by Liendo in \cite{Lie10a}, generalizing the former classification by Flenner and Zaidenberg for normal complex affine $\CC^{\star}$-surfaces (see  \cite{FZ05}). A next step to solve this problem was completed in \cite{LL16} for perfect ground fields. 
We treat here the remaining case, where the base field is possibly imperfect (see Theorem \ref{main} for the main result).

The investigation of the $\G_{a}$-actions on algebraic varieties is of highest interest in the research field called affine geometry where most of the classical problems can be reformulated in terms of $\G_{a}$-actions satisfying some properties. The systematic method employed in this paper could serve as well for describing the automorphism groups of certain complete varieties. This might be performed, for instance, by using the Cox ring theory or by interpreting the $\G_{a}$-action as an integrable vector field, see \cite{AHHL14, DL16, ALS19} for the characteristic-zero case. These developments are still open in positive characteristic.
Special attention should be paid to the case of imperfect base fields, as one may describe some interesting algebraic groups (that only show up in this case) as 
automorphism groups of projective varieties such as the pseudo-reductive groups \cite{BGP10} or the pseudo-abelian varieties \cite{Tot13}.

Let us now introduce some notation in order to state our main results. In the entire paper we fix a split algebraic torus $\TT\simeq  \G_{m}^{n}$ over $k$. We make the convention that a \emph{$\TT$-variety} is a normal variety $X$ equipped with a faithful $\TT$-action. The \emph{complexity} of $X$ is then defined as the transcendence degree of the field extension $k(X)^{\TT}/k$, where the subfield $k(X)^{\TT}\subseteq k(X)$ consists of invariant functions. Here we use the usual notation such
as $k[X] = \Gamma(X, \mathcal{O}_{X})$ for the $k$-algebra of regular global functions and $k(X)$ for the residue field of the generic point of $X$. 

Assuming $X$ to be affine, a $\G_{a}$-action on $X$ is equivalent to having a \emph{locally finite iterative higher derivation} (LFIHD) on $k[X]$, see \cite{Miy68}. This is a family of $k$-linear operators 
$$\partial = \{\partial^{(i)}: k[X]\rightarrow k[X]\}_{i\geq 0}$$
respecting the following conditions. $(i)$ $\partial^{(0)}$ is the identity; $(ii)$ for all $b_{1}, b_{2}\in k[X]$ and $i\in\ZZ_{\geq 0}$, 
$$\partial^{(i)}(b_{1}\cdot b_{2}) = \sum_{i = i_{1} + i_{2}} \partial^{(i_{1})}(b_{1})\cdot \partial^{(i_{2})}(b_{2});$$
$(iii)$ for any $b\in k[X]$ we have $\partial^{(j)}(b) = 0$ for $j\gg 0$; $(iv)$ for all $u,v\in\ZZ_{\geq 0}$, 
$$\partial^{(u)}\circ  \partial^{(v)} = \binom{u+v}{u}\partial^{(u+v)}.$$
Moreover, the datum of a $\TT$-action on $X$ translates into an $M$-grading on the $k$-algebra $k[X]$, where $M$ denotes the character group of the torus $\TT$. Note that finite type normal $M$-graded algebras admit
combinatorial descriptions in terms of polyhedral divisors (notion invented by Altmann
and Hausen; see \cite{AH06, Tim08, Lan15}). The idea in \cite{Lie10a} (and this is the viewpoint of the present paper) is to classify the $\G_{a}$-actions in question
using this combinatorial approach. 

A $\G_{a}$-action is said to be \emph{normalized} by the $\TT$-action if, for the corresponding LFIHD $\partial$, there is a lattice vector $e\in M$ (called the \emph{degree} of $\partial$) such that the linear maps $\partial^{(j)}$ are 
homogeneous of degree $je$. In other words, this means that any  homogeneous element in $k[X]$ of degree
$m\in M$ is sent to a homogeneous element of degree $m +je\in M$ by the map  $\partial^{(j)}$ for every integer $j\in\ZZ_{\geq 0}$. Using the Leibniz rule, that is Condition $(ii)$ before, the LFIHD $\partial$ extends to a sequence of linear operators
on the function field $k(X)$ satisfying Conditions $(i), (ii), (iv)$ (see e.g. \cite[Lemma 2.5]{LL16}). We say, in addition, that the $\G_{a}$-action normalized by the torus action is \emph{vertical} (or of \emph{fiber type})
if $\partial^{(j)}(k(X)^{\TT}) = 0$ for any $j\in \ZZ_{> 0}$, where $\partial$ means the extension on the function field $k(X)$. Otherwise, the normalized $\G_{a}$-action is called \emph{horizontal}. Over any field, the vertical $\G_{a}$-actions on complexity-one affine $\TT$-varieties were described in \cite[Section 4]{LL16} (see also \cite{Lie10b}). Therefore, it remains to look at the horizontal case. Our first 
main result (see \ref{main}) can be stated as follows. Here the combinatorial equipment for describing the horizontal $\G_{a}$-actions are the coherent families, see Definition \ref{coherent} for more details. 
\begin{theorem}\label{main1}
Let $X$ be a complexity-one affine $\TT$-variety described by a polyhedral divisor $\D$
over a regular curve $C$.  Then the presence of a horizontal $\G_{a}$-action on $X$ implies that $C = \A^{1}_{k}$ or $C = \P^{1}_{k}$, and in this case, the map $\theta \mapsto \partial_{\theta}$ induces a bijection between the set of coherent families $\theta = (\tilde{\D}, e, \underline{s}, \underline{\lambda})$ on $\D$ and the set of horizontal LFIHDs on the $k$-algebra $k[X] = A[C, \D]$.
\end{theorem}

The key observation for proving Theorem \ref{main1} is that the existence of such a horizontal $\G_{a}$-action automatically implies that the complexity-one affine $\TT$-variety is geometrically integral
over $k$ (see Lemma \ref{cor-geomn}) and therefore one may extend the scalars to an algebraic closure $\bar{k}$. While the proof of our main result (see Theorem \ref{main}) boils down to understanding this field extension problem,
it is worthwhile mentioning that these $k$-varieties are not in general geometrically normal. For instance, one may take the ones that have a regular non-smooth affine global quotient. But, as observed in Lemma \ref{cor-geomn}, such examples do not fulfill the condition to have any horizontal $\G_{a}$-action. Nevertheless, we exhibit in \ref{final-ex} a non-geometrically normal example of complexity-one affine $\TT$-variety that possesses a horizontal $\G_{a}$-action.

As a further matter, the inseparable degree of the closed points of the rational quotient  appears in the description of the horizontal $\G_{a}$-action as a new numerical invariant (see Definition \ref{coherent} $(v)$).
Finally, we show the following intermediate result (see \ref{surface} and consult \cite[Section 2]{FZ03a}, \cite[Theorems 3.3 and 3.16]{FZ05}, \cite[Corollary 5.6]{LL16} for perfect ground fields) on the geometric structure of normal affine $\G_{m}$-surfaces with horizontal $\G_{a}$-action 
that gives rise to a positive grading.
\begin{theorem}\label{surface1}
Any (normal) affine $\G_{m}$-surface over arbitrary field that is not hyperbolic and having a horizontal $\G_{a}$-action is a toric surface.
\end{theorem}

\begin{remark}
We will suppose that for any algebraic variety $V$ 
over $k$, the field $k$ is \emph{algebraically closed} in the function field $k(V)$. That means that all elements of $k(V)$ that are solutions to a one-variable polynomial in $k[t]$ belong
to $k$. This assumption implies that $V$ is geometrically irreducible (see \cite[\S 3.2, Corollary 2.14$(d)$]{Liu02}) and allows one 
to simplify the statements of the present paper. It is not restrictive since, in the general case, one may change the base field $k$ by its algebraic closure in $k(V)$.   
We will also let $V_{K} = V\times_{\spec\, k}\spec\, K$ for any field extension $K/k$. 
Finally, the letter $p$ will denote the characteristic exponent of the field $k$.
\end{remark}
\section{Classification of horizontal $\G_{a}$-actions}
\subsection{Torus actions and polyhedral divisors}
We fix a complexity-one affine $\TT$-variety $X$ defined over $k$. The torus action on $X$ involves considering a combinatorial description which was, in particular, initiated by Mumford \cite{KKMS73}, Demazure \cite{Dem88}, Timash\"ev \cite{Tim97, Tim08}, Flenner-Zaidenberg \cite{FZ03b}
and Altmann-Hausen \cite{AH06}. We will adopt the notation used in \cite{AH06} (see \cite{Lan15} for a version over any field). Especially, the letter $M$ stands for the character lattice of the torus $\TT$, the space $N:= {\rm Hom}(M,\ZZ)$ is the lattice of one-parameter subgroups and $N_{\QQ} := \QQ\otimes_{\ZZ}N, M_{\QQ} := \QQ\otimes_{\ZZ}M$ are the associated $\QQ$-vector spaces. We write $\langle \cdot , \cdot \rangle$ for the duality bracket between $M_{\QQ}$ and $N_{\QQ}$. The $M$-graded algebra $A:= k[X]$ admits a decomposition (cf. \cite[Theorem 3.4]{AH06}, \cite[Theorem 0.2]{Lan15})
$$A = A[C, \D] := \bigoplus_{m\in \sigma^{\vee}\cap M}H^{0}(C, \mathcal{O}_{C}(\lfloor \D(m)\rfloor))\chi^{m},$$
where $C$ is a regular curve over $k$ and the subset $\sigma\subseteq N_{\QQ}$ is a polyhedral cone with the property that the dual $\sigma^{\vee}\subseteq M_{\QQ}$ is
full-dimensional. The letter $\D$ denotes a formal sum
$$\D = \sum_{y\in C}\D_{y}\cdot [y]$$ on the closed points of $C$ which defines a \emph{$\sigma$-polyhedral divisor}, that is, we ask that each subset $\D_{y}\subseteq N_{\QQ}$ is a Minkowski sum of a polytope with $\sigma$ and $\D_{y} = \sigma$
for almost all closed point $y\in C$. One can evaluate the polyhedral divisor at the vector $m\in \sigma^{\vee}$
via the equality 
$$\D(m) = \sum_{y\in C}\min_{v\in \D_{y}}\langle m, v \rangle\cdot [y],$$
which actually gives  a $\QQ$-divisor on $C$. Finally, in each graded piece, a Laurent monomial $\chi^{m}$ is attached for keeping track of the degree of every homogeneous element. Note that some positivity assumptions are required on the evaluations $\D(m)$ in order to 
have a perfect dictionary between complexity-one affine $\TT$-varieties and polyhedral divisors over regular curves (see \cite[Definition 0.1]{Lan15} for the details). In the case where the base curve $C$ is affine, no condition is required, whereas for $C = \P ^{1}_{k} $ (which is a case of main interest in this article), the positivity is equivalent to ask the inclusion of polyhedra
$${\rm deg}\, \D := \sum_{y\in C} [\kappa_{y} : k]\cdot \D_{y}\subsetneq \sigma,$$
where $\kappa_{y}$ is the residue field of the closed point $y\in C$.
\\

\subsection{Preliminary results on $\G_{a}$-actions}
We start by generalizing some results over any field which are originally derived from the fundamental work of Liendo in \cite[Section 3.2]{Lie10a}. As usual, the letter $A = A[C, \D]$ will stand for the algebra
of regular functions of our complexity-one $\mathbb{T}$-variety $X$.
For every homogeneous LFIHD $\partial$ on $A = k[X]$, recall that the subset
$\ker(\partial) := \bigcap_{i = 1}^{\infty} \ker(\partial^{(i)})$
denotes the \emph{kernel} of $\partial$; this an $M$-graded subring of $A$.
\begin{lemma}\label{cor-geomn}
Assume that the homogeneous LFIHD $\partial$ on $A = A[C, \D]$ is horizontal. Then the following assertions hold.
\begin{itemize}
\item[(i)] The kernel of $\partial$ is a semigroup algebra, that is 
$$\ker(\partial) = \bigoplus_{m\in \omega\cap L} k \varphi_{m}\chi^{m},$$
where $\varphi_{m}\in k(C)^{\star}$. The set $\omega\subseteq M_{\QQ}$ is a full-dimensional polyhedral cone and $L\subseteq M$ is a sublattice such that the quotient $M/L$ is a finite abelian group.
\item[(ii)] If $C$ is projective, then $C\simeq \P_{k}^{1}$.
\item[(iii)] The variety $X$ is geometrically integral over $k$.
\item[(iv)] If $C$ is affine, then $C\simeq \A_{k}^{1}$ and in this case, 
the equality $\div(\varphi_{m}) + \D(m) = 0$ holds true for any $m\in\omega\cap L$.
\end{itemize}
\end{lemma}
\begin{proof}
Assertions $(i), (ii)$ have been proven in \cite[Lemma 5.2]{LL16}. 

$(iii)$ Using the existence of a local slice for the $\G_{a}$-action (see \cite[Lemma 1.5, p20]{Miy71}) and 
Assertion $(i)$, we may find a transcendent element $x_{n+1}$ over $\ker(\partial)$ such 
that 
$$k(X) = k(x_{1} = \varphi_{e_{1}}, \ldots, x_{n} = \varphi_{e_{n}}, x_{n+1}),$$
where $(e_{1}, \ldots, e_{n})$ is a basis of the lattice $L$. Note that $x_{1}, \ldots, x_{n+1}$ are algebraically independent over $k$. Hence for any field extension $K/k$ we have that
$$A\otimes_{k}K\subseteq k(x_{1}, \ldots, x_{n+1})\otimes_{k} K \simeq K(x_{1}, \ldots, x_{n+1}),$$
proving that $X = \spec\, A$ is geometrically integral over $k$.

$(iv)$ Since finite fields are perfect, using \cite[Lemma 5.2 (iii)]{LL16}, we may assume 
that the cardinality of $k$ is infinite. By virtue of \cite[Lemma 1.5, p20]{Miy71}, one can find
a homogeneous element $f$ in $\ker(\partial)$ and a transcendent homogeneous element $x$ over $\ker(\partial)$
such that $A_{f} = \ker(\partial)_{f}[x]$. From the inclusion $k[C] = A^{\TT}\subseteq A$, we get 
a dominant morphism
$\gamma: V\times \A^{1}_{k}\rightarrow C$, where $V = \spec\, \ker(\partial)_{f}$ is a $k$-variety with a $\TT$-action and having a dense open orbit. Assume (toward a contradiction) that for every $v\in V(k)$, there exists a closed schematic point $y_{v}\in C$ such that
$\{v\}\times \A^{1}_{k}\subseteq \gamma^{-1}(y_{v})_{\rm red}$. Since $k$ is infinite, one observes that if $v\in V(k)$ belongs to the open orbit, then $\TT(k)\cdot v$ is dense in $V$.  
Using that $\gamma$ is $\TT$-invariant, 
$$V\times \A^{1}_{k} = \overline{ \TT(k)\cdot (\{v\}\times \A_{k}^{1})}\subseteq \gamma^{-1}(y_{v})_{\rm red},$$
which contradicts the dominance of $\gamma$. We conclude that there exists $v\in V(k)$ such that the map
$$\gamma_{|\{v\}\times\A^{1}_{k}}: \{v\}\times\A^{1}_{k}\rightarrow C$$
is dominant. 

As $k(C) = k(\P^{1}_{k})$, the regular affine curve $C$ is an open subscheme 
of $\P^{1}_{k}$. 
By the argument before we have a 
dominant morphism $C_{1} := \A^{1}_{k}\rightarrow C$ which extends to the completions into a proper morphism 
$\bar{\gamma}: \P^{1}_{k}\rightarrow \P^{1}_{k}$. Let $y_{\infty}$ be the $k$-rational point such that $\P^{1}_{k}\setminus \{y_{\infty}\} =C_{1}$. Then the surjectivity of $\bar{\gamma}$
implies that $\{\bar{\gamma}(y_{\infty})\}$ is the complement of $C$ in its regular completion. We conclude that $C\simeq \A^{1}_{k}$. The last claim of $(iv)$ is done in \cite[Lemma 5.4 (i)]{LL16}.
\end{proof}
According to the previous result, we have that $C = \A^{1}_{k}$ or $ C = \P^{1}_{k}$. Let $t$
be a variable over $k$ such that $\A^{1}_{k} = \spec\, k[t]$ and $k(\P^{1}_{k}) =k(t)$. We write $\infty$ for the $k$-rational point of $\P^{1}_{k}$ satisfying $\P^{1}_{k}\setminus\{\infty\} = \A^{1}_{k}$. In particular, the principal divisor $\div(t)$ on $\P^{1}_{k}$ is equal to $[0] - [\infty]$.
For a closed point $y\in\A^{1}_{k}$ we denote by $q_{y}(t)\in k[t]$ the associated monic irreducible polynomial. We may write $q_{y}(t) = \tilde{q}_{y}(t^{p^{\ell}}),$ where 
$\tilde{q}_{y}(t)\in k[t]$ is a polynomial with nonzero derivative with respect to the variable $t$. We call the number $\varepsilon_{y}:= p^{\ell}$ the \emph{inseparable degree} of $y$. In the ring $\bar{k}[t]$, we have the decomposition 
$$q_{y}(t) = \prod_{i=1}^{s_{y}}(t- \alpha_{i, y})^{\varepsilon_{y}},$$
where the $\alpha_{i, y}\in \bar{k}$ are pairwise distinct. Note that if $s_{y} = 1$, then we say that $y$ is \emph{purely inseparable}. We also set
$$\Dbar = \sum_{y\in C} \sum_{i =1}^{s_{y}}\varepsilon_{y}\D_{y} \cdot [\alpha_{i,y}],$$
where here $(\Dbar)_{\infty} = \D_{\infty}$ if $C = \P_{k}^{1}$; this defines a $\sigma$-polyhedral divisor over $C_{\bar{k}}$. We introduce similar notations for polyhedral divisors 
over a non-empty open subscheme of $\P^{1}_{k}$.
\begin{lemma}\label{ext} 
Assume that the complexity-one $\TT$-variety $X_{0}$ comes from a polyhedral divisor $\D_{0}$
over a non-empty open subscheme $C_{0}\subseteq \P^{1}_{k}$.
The normalization of $X_{0, \bar{k}} = \spec\, A[C_{0}, \D_{0}]\otimes_{k}\bar{k}$  is described by the polyhedral divisor $\D_{0, \bar{k}}$. Moreover, via this description, we have the equality $k(X_{0})\cap A[C_{0, \bar{k}}, \D_{0, \bar{k}}] = k[X_{0}]$. 
\end{lemma}
\begin{proof}
The first part of the lemma is obtained by choosing a finite system of homogeneous generators
$a_{1}, \ldots, a_{s}$ and by determining the normalization of the ring $\bar{k}[a_{1}, \ldots, a_{s}]$ via \cite[Th\'eor\`eme 2.4]{Lan13} (see \cite[Theorems 2.5, 3.5]{Lan15} for the version over any field). For the second part, we obviously have that $k[X_{0}]\subseteq  k(X_{0})\cap A[C_{0,\bar{k}}, \D_{0, \bar{k}}]$. Let now $\beta\in k(X_{0})\cap A[C_{0, \bar{k}}, \D_{0, \bar{k}}]$. Then there exists a finite field extension $K/k$ such that $\beta\in R_{K}$, where $R_{K}$ is the normalization of $k[X_{0}]\otimes_{k}K$. As $R_{K}$ is a finite type module over $k[X_{0}]$, the element $\beta$ is integral over $k[X_{0}]$. Using that $k[X_{0}]$ is a normal ring, we conclude that $\beta\in k[X_{0}]$, proving the lemma. 
\end{proof}
\begin{lemma}\label{ratinfty}
Assume that our algebra $A  =k[X]$ has a horizontal LFIHD $\partial$ and that $C = \P^{1}_{k}$.
Then there exists a $k$-rational point $y_{\infty}\in C(k)$ such that for any $m\in\omega\cap L$
the effective $\QQ$-divisor $\D(m)  + \div(\varphi_{m})$ has at most $y_{\infty}$ in its support. Here $\varphi_{m}$ is the regular function from Lemma \ref{cor-geomn}$(i)$.
\end{lemma}
\begin{proof}
Changing $\partial$ by $\{\xi^{i}\cdot \partial^{(i)}\}_{i\geq 0}$, where $\xi$ is a homogeneous element of $\ker(\partial)$, we may assume that the degree of $\partial$ belongs to $\omega$.
Since from \cite[Lemma 5.4 (v)]{LL16} the effective $\QQ$-divisor $\Dbar(m)  + \div_{C_{\bar{k}}}(\varphi_{m})$ over $C_{\bar{k}}$ is supported in at most one point, one concludes that $\D(m)  + \div_{C}(\varphi_{m}) = \alpha\cdot [y_{\infty }]$ for some $\alpha\in \ZZ_{\geq 0}$ and some purely inseparable closed point $y_{\infty}\in C$. Let us prove that $y_{\infty}$ is a $k$-rational point. Let $B:= A[C', \D_{|C'}],$ where $C' = C\setminus \{y_{\infty }\}$.
Consider moreover  the normalization $B'$ of the ring $B\otimes_{k}\bar{k}$. Then, with respect to 
a local parameter $s$ over $\bar{k}$ (such that $\bar{k}(C_{\bar{k}}) = \bar{k}(s)$), the ring $B'$ is indeed the normalization of the ring $A[C_{\bar{k}}, \Dbar][s]$. According to \cite[Lemma 5.5 (ii)]{LL16}, the LFIHD $\partial$ on $A$ extends to one on $B'$. As $k(X)\cap B'= B$ (use Lemma \ref{ext}), we
remark that $\partial$ extends to a horizontal LFIHD on $B$. This forces $C'$ to be isomorphic to $\A^{1}_{k}$ (see Lemma \ref{cor-geomn}(iv)) and therefore $y_{\infty}$ to be a $k$-rational point. 
\end{proof}
Until now we may assume that the $k$-rational point $y_{\infty}$ in Lemma \ref{ratinfty} is the point $\infty\in \P^{1}_{k}$ corresponding to the local parameter $t$ considered before. With this in hand,
we obtain the following corollary, which is a generalization of \cite[Lemma 5.5 (ii)]{LL16}
over any field.
\begin{corollary}\label{passtoaffine}
Assume that $C = \P^{1}_{k}$ and that $A = k[X]$ has a horizontal LFIHD $\partial$ with degree belonging to $\omega$. Then the normalization of $A[t]$, that is, the algebra
consisting of elements of the function field of $A[t]$ that are integral over $A[t]$, identifies with $A[\A^{1}_{k}, \D_{|\A^{1}_{k}}]$. Moreover, under this identification, $\partial$ extends to a horizontal LFIHD on the algebra $A[\A^{1}_{k}, \D_{|\A^{1}_{k}}]$.
\end{corollary} 
\begin{proof}
This litteraly follows from the proof of \cite[Lemma 5.5 (ii)]{LL16}. In the argument of the proof of \emph{loc. cit.}, we only need to replace \cite[Lemma 5.2(i)]{LL16} by Lemma \ref{cor-geomn}(iv)
and \cite[Lemma 5.2(ii)]{LL16} by Lemma \ref{ratinfty}.
\end{proof}

An affine $\G_{m}$-surface is \emph{hyperbolic} if its $\ZZ$-graded algebra is not positively graded. 
The following theorem is a generalization of \cite[Theorems 3.3 and 3.16]{FZ05} and \cite[Corollary 5.6]{LL16}.

\begin{theorem}\label{surface}
Any (normal) affine $\G_{m}$-surface $X$ over a field $k$ which is not hyperbolic and admitting a horizontal $\G_{a}$-action is toric. More precisely, assume that
$X$ is described by a polyhedral divisor $\D$ over a regular curve $C$. Then $C = \A^{1}_{k}$ or $C = \P^{1}_{k},$ and the fractional part $\{\D(1)\}$ of the $\QQ$-divisor $\D(1)$ is supported in at most one $k$-rational point if $C = \A^{1}_{k}$ and in at most two $k$-rational points if  $C = \P^{1}_{k}$.
\end{theorem}
\begin{proof}
It suffices to show the second statement, as the toridicity of our surface is a consequence of this (see e.g. the end of the proof of \cite[Corollary 5.6]{LL16}).
The case $p =1$ is treated in \cite[Corollary 5.6]{LL16}. So we assume in the entire 
proof that $p>1$. We first look at the case where $C = \A^{1}_{k}$. Let $d\in \ZZ_{>0}$ such  that $\D(d)$ is an integral divisor and consider a rational function $f\in k(t)^{\star}$ that generates 
the $k[t]$-module $H^{0}(C, \mathcal{O}_{C}(\D(d)))$. Without loss of generality, we may assume that $f\in k[t]\setminus\{0\}$ is a monic polynomial. Let $B$ be the normalization of $A[\sqrt[d]{f}\chi]$, where $A = k[X]$. Taking $d$ large enough if necessary, we may assume $f\chi^{d}\in \ker(\partial)$ (see Lemma \ref{cor-geomn}(iv)) and hence according to \cite[Corollary 2.6]{LL16} the LFIHD $\partial$ extends to a horizontal one on $B$. Now $B_{0} = B^{\G_{m}}$ is the normalization of $k[t, \sqrt[d]{f}]$ and also a polynomial algebra of one variable over $k$ (compare with Lemma \ref{cor-geomn}(iv)). Moreover, if $S$ is the normalization
of $B\otimes_{k}\bar{k}$, then $S$ has a horizontal $\G_{a, \bar{k}}$-action and $S_{0} = S^{\G_{m}}$ (which is the normalization  of $\bar{k}[t, \sqrt[d]{f}]$) is a polynomial ring too. Using that $S$ is factorial and that $S^{\star} = k^{\star}$, we must have $f = (t-\mu)^{r}$ for some $\mu\in\bar{k}$ and $r\in \ZZ_{>0}$. In addition, $f$ belongs to $k[x]$ and so we write 
$f = (t^{p^{u}}-\lambda)^{v}$ for some $v\in \ZZ_{\geq 0}\setminus p\ZZ$, $u\in \ZZ_{\geq 0}$
and some $\lambda = \mu^{p^{u}}\in k$.

This implies that 
$$k\left[t, \sqrt[d]{f}\right] \simeq k[x_{1}, x_{2}]/\left(x_{1}^{d} - (x_{2}^{p^{u}} - \lambda)^{v}\right).$$ Assume that 
$$\lambda\not\in k^{p^{u}} := \left\{x^{p^{u}}\, |\, x\in k\right\}.$$ If $p$ divides $d$, then any geometric point of $C_{0}:= \spec\, k\left[t, \sqrt[d]{f}\right]$ is singular and this contradicts that $\tilde{C}:= \spec\, B_{0}$ is an affine line over $k$. So $p$ does not divides $d$.
Summing $\D(1)$ with a principal divisor, we may assume that $d\geq 2$ and $v<d$.
Let us write $v/d$ as an irreducible fraction $e/d'$. Then we claim that the normalization
$\tilde{C}$ of $C_{0}$ has algebra of regular functions equal to $R:= k[x, y= (x^{p^{u}}-\lambda)^{1/d'}]$, i.e., $\tilde{C}$ is the plane curve defined by the equation
$y^{d'} = x^{p^{u}}-\lambda.$ Indeed, this is clear that $R\subseteq k(C_{0})$ is integral
over $k[C_{0}]$. So we only need to check that $R$ is a normal ring. Using the Jacobian criterion over $\bar{k}$, this amounts to show that $R$ is regular at the prime ideal
$\mathfrak{p} = (y)$. Let $s = (x - \mu)^{1/d'}$ and consider the field extension $k(C_{0})\subseteq \bar{k}(s)$ where we get the parametrization $x= s^{d'} + \mu$ and $y = s^{p^{u}}$. 
Let $\nu_{0}$ be the discrete valuation on $\bar{k}(s)$, trivial on $\bar{k}$ and 
satisfying $\nu_{0}(s) =1$. Denote by $\nu$ the restriction of the valuation $\frac{1}{p^{u}}v_{0}$ to the subfield $k(C_{0})$.  Remarking that
$$I_{\nu} := \{f\in R\setminus\{0\}\,|\, \nu(f)>0\}\cup\{0\}$$
is a proper ideal that contains $\mathfrak{p}$, we have $I_{\nu} = \mathfrak{p}$. Therefore
the local ring $R_{\mathfrak{p}}$ coincides with the valuation ring associated with $\nu$
and we conclude that $\tilde{C} = \spec\, R$ is regular. Note that $\tilde{C}$ is not smooth (according to the Jacobian criterion for $\tilde{C}_{\bar{k}}$). This contradicts the fact that $\tilde{C}$ is the affine line over $k$. 
Finally, $\lambda \in k^{p^{u}}$
and $\{\D(1)\}$ is supported in at most one $k$-rational point (by observing that $\div(f) +\D(d) =0$).

We pass to the case where $C = \P^{1}_{k}$. Changing $\partial$ by $\{\xi^{i}\cdot\partial^{(i)}\}_{i\geq 0}$, where $\xi$ is a homogeneous element of $\ker(\partial)$, we may assume that the degree $e$ of $\partial$ is positive.
Using Corollary \ref{passtoaffine}, the LFIHD
$\partial$ extends on the normalization $A[\D_{|\A^{1}_{k}}, \A_{k}^{1}]$ of $A[t]$ into
a horizontal one. Thus, we end the proof of the theorem by using the previous case where $C$ was assumed to be affine.
\end{proof}
As a consequence of Theorem \ref{surface} we get the following result. We will here assume that $C$ is equal to $\A^{1}_{k}$ or $\P^{1}_{k}$.
\begin{corollary}\label{l-linea}
Let us suppose that $A = k[X]=  A[C,\D]$ has a horizontal LFIHD $\partial$.
Then the  following statements hold.
\begin{itemize}
\item[(i)] The cone $\omega\subseteq M_{\QQ}$ introduced in Lemma  \ref{cor-geomn}
is a maximal subcone of $\sigma^{\vee}$ in which the restriction of the map $m\mapsto \D(m)$ 
when $C$ is affine, or of the map $m\mapsto \D(m)_{|\A^{1}_{k}}$ when $C$ is projective, to it is linear.
\item[(ii)] Set 
$$A_{\omega} = \bigoplus_{m\in \omega\cap M}H^{0}(C, \mathcal{O}_{C}(\lfloor \D(m) \rfloor))\chi^{m}$$
and let $\tau = \omega^{\vee}\subseteq N_{\QQ}$ be the dual cone. Then $A_{\omega}$ is isomorphic to $A[C, \D_{\omega}]$ as $M$-graded algebras, where $\D_{\omega}$ is the $\tau$-polyhedral divisor over the curve $C$ satisfying the following conditions.
\begin{itemize}
\item[$(a)$] $\D_{\omega} = (v + \tau)\cdot [0]$ 
for some $v\in N_{\QQ}$, whenever $C$ is affine.
\item[$(b)$]  $\D_{\omega} = (v + \tau)\cdot [0]  + \D_{\omega, \infty}\cdot [\infty]$ for some $v\in N_{\QQ}$ such that  $v + \D_{\omega, \infty}\subsetneq \tau$, whenever $C$ is projective.
\end{itemize}
\end{itemize}
\end{corollary}
\begin{proof}
The proof is similar to that of \cite[Lemmata 3.18, 3.23]{Lie10a}. We include the argument here for the convenience of the reader.

(i) According to Lemma \ref{cor-geomn}(iv) and Lemma \ref{ratinfty},  we only need to prove the required maximality property. Since the case $C =\P_{k}^{1}$ is similar, we may assume that $C =\A_{k}^{1}$. Let $\omega_{0}$ be a subcone
containing $\omega$ 
where the evaluation map $m\mapsto \D(m)$ restricted to it is linear. Let us show that $\omega = \omega_{0}$. Pick $m\in \omega_{0}\cap L$ such that $\D(m)$ is an integral divisor 
and let $f_{m}$ be a generator of the $k[t]$-module $H^{0}(C, \mathcal{O}_{C}(\lfloor \D(m) \rfloor))$. For an element $m'\in \omega\cap L$ such that $m + m'\in \omega\cap L$ we have 
$f_{m}\cdot \varphi_{m'}\chi^{m + m'}\in \ker(\partial)$. As $\ker(\partial)$ is factorially 
closed (cf. \cite[Lemma 2.1 (a)]{CM05}), we see that $f_{m}\chi^{m}\in \ker(\partial)$ and therefore $m\in\omega\cap L$. This shows $(i)$.

(ii) Using Corollary \ref{passtoaffine}, we may assume that $C =\A_{k}^{1}$ and $e\in \omega$. Letting $\ell\in \omega\cap L$ we observe that the subalgebra
$$\bigoplus_{r\geq 0}H^{0}(C, \mathcal{O}_{C}(\lfloor \D(r(\ell + e))\rfloor))\chi^{r(\ell + e)}$$
is stable by the $\G_{a}$-action coming from the LFIHD $\{\varphi_{\ell}^{i}\cdot \partial^{(i)}\}_{i \geq 0}$. Therefore from Theorem \ref{surface}, the $\QQ$-divisor
$\{\D(\ell +e)\}$ is supported in at most one $k$-rational point. Now we remark that
$\{\D(\ell +e)\} = \{\D(\ell' +e)\}$ for all $\ell, \ell'\in \omega\cap L$ since the difference between $\D(\ell + e)$ and $\D(\ell' + e)$ is equal to ${\rm div}(\varphi_{\ell'}/\varphi_{\ell})$. As the subset $\{e\}\cup L$ generates $M$, we conclude
that assertion $(a)$ of the lemma holds.
\end{proof}

\subsection{Coherent families and Demazure roots} In this section, we again suppose that $C  = \A^{1}_{k}$ or $C = \P^{1}_{k}$. We want to give a
combinatorial  description of the horizontal $\G_{a}$-actions as in
\cite[Section 1.4]{AL12}. We start with the following definition.

\begin{definition}\label{defcolcol}
A collection $$\tilde{\D} = (\D, (v_{y})_{y\in C'}, y_{0}),$$ which in particular encompasses
the polyhedral divisor $\D$ over $C$ describing our complexity-one affine $\TT$-variety $X$ and a vertex $v_{y}$ of $\D_{y}$ for each schematic closed point $y$ belonging to a subset $C'$ of $C$, is called a \emph{coloring} if 
the following conditions are fulfilled. 
\begin{itemize}
\item[(i)] There is $y_{\infty} \in C(k)$ if $C = \P^{1}_{k}$ which allows to define the curve $C'$ as $C\setminus\{y_{\infty}\}$. Otherwise, we let $C' = \A^{1}_{k}$.
\item[(ii)] We have $y_{0}\in C'(k)$ and $v_{y}\in N$ for any $y\in C'\setminus\{y_{0}\}$. 
\item[(iii)] The element $v_{\deg}:= \sum_{y\in C'}[\kappa_{y}:k]\cdot v_{y}$
is a vertex of the polyhedron 
$$\deg\, \D_{|C'}:= \sum_{y\in C'} [\kappa_{y}:k]\cdot \D_{y},$$
where $\kappa_{y}$ stands for the residue field of $y$.
\end{itemize}
The vectors $v_{y}$ are called the \emph{colored vertices} of $\D$. 
\end{definition}

We now recall the notion of Demazure roots which was initially introduced in \cite[Section 4.5]{Dem70} for studying Cremona groups.
It was then extended by Liendo in \cite[Section 2]{Lie10a} for singular affine toric varieties.
\begin{definition}
Let $\sigma_{0}\subseteq N_{\QQ}$ be a polyhedral cone with $0$ as vertex. An element 
$e\in M$ is called a \emph{Demazure root} of the cone $\sigma_{0}$ with distinguished one-dimensional face $\rho_{e}$ if 
\begin{itemize}
\item[(i)] $\langle e, v\rangle \geq 0$ for any $v\in \sigma_{0}\setminus \rho_{e}$ belonging 
to a one-dimensional face of $\sigma_{0}$, and
\item[(ii)] $\langle e, \mu_{e}\rangle = -1$, where $\mu_{e}$ is the generator of the semigroup  $(\rho_{e}\cap N, +)$.
\end{itemize}
Considering the semigroup algebra $$k[\sigma^{\vee}_{0}\cap M] = \bigoplus_{m\in\sigma_{0}^{\vee}\cap M} k\chi^{m},$$ it has been proved in \cite[Theorem 3.5]{LL16} (see \cite[Section 4.5]{Dem70}, \cite[Theorem 2.7]{Lie10a} for the characteristic zero case) that any homogeneous LFIHD on this algebra
(up to a constant) is described by a Demazure root of $\sigma_{0}$ and vice-versa. If $e\in M$ is such a Demazure root, then the corresponding LFIHD $\partial_{e}$ is defined via the formula
$$\partial^{(i)}_{e}(\chi^{m}) = \binom{\langle m, \mu_{e}\rangle}{i}\cdot \chi^{m+ie}\text{ for all }i\in\ZZ_{\geq 0} \text{ and }m\in \sigma^{\vee}_{0}\cap M .$$
\end{definition}
For a coloring $\tilde{\D}$ (see Definition \ref{defcolcol}), the \emph{associated cone} $\omega$ (which will play the role of the weight cone of the kernel of the corresponding horizontal LFIHD) is the polyhedral cone whose dual $\tau\subseteq N_{\QQ}$ is spanned by $\deg\, \D_{|C'} - v_{\deg}$. Also, we denote by $\tilde{\omega}\subseteq M_{\QQ}\times \QQ$ the polyhedral cone 
whose the dual $\tilde{\tau}$ is spanned by $(\tau, 0), (v_{y_{0}},1)$ if $C = \A^{1}_{k}$
and by $(\tau, 0), (v_{y_{0}}, 1), (\D_{y_{\infty}} + v_{\deg} - v_{y_{0}} + \tau, -1)$ if $C = \P^{1}_{k}$.
\\

We then introduce the main tool for describing horizontal $\G_{a}$-actions on complexity-one affine $\TT$-varieties (see \cite[Definition 1.9]{AL12} for the classical case). Recall 
that $p$ is the characteristic exponent of the base field $k$.
\begin{definition}\label{coherent}
A family $\theta = (\tilde{\D}, e , \underline{s}, \underline{\lambda})$ is said to be \emph{coherent}
if first
\begin{itemize}
\item[$(i)$] $\tilde{\D} = (\D, (v_{y})_{y\in C'}, y_{0})$ is a coloring of $\D$,
\item[$(ii)$] $e$ is a lattice vector of $M$,
\item[$(iii)$] $\underline{s}$ is a family of positive integers $(s_{1}, \ldots, s_{r})$ such that $s_{1}<\ldots < s_{r}$ and with the condition that $r=1 = s_{1}$ whenever $p =1$. Moreover, we ask that
$$\tilde{e}_{i} := \left(p^{s_{i}}e, -\frac{1}{d}-\langle p^{s_{i}}e, v_{y_{0}}\rangle \right)  \text{ for } i=1, \ldots, r,$$
is a Demazure root of the cone $\tilde{\tau}\subseteq N_{\QQ}\times \QQ$
with distinguished ray $\QQ_{\geq 0}\cdot (v_{y_{0}}, 1)$, where $d\in \ZZ_{>0}$ is the smallest element such that $dv_{y_{0}}\in N$,
\item[(iv)] and finally, $\underline{\lambda}$ is a sequence $(\lambda_{1}, \ldots, \lambda_{r})$ of elements of $k^{\star}$.
\end{itemize} 
Together, they satisfy the following constraints.
\begin{itemize}
\item[$(v)$] We have that
$$\varepsilon_{y}\cdot p^{u}\langle p^{s_{1}}e, v\rangle \geq 1 + \varepsilon_{y}\cdot  p^{u}\langle p^{s_{1}}e, v_{y}\rangle$$
for any $y\in C'\setminus \{y_{0}\}$ and any uncolored vertex $v$ of $\D_{y}$, where $d = \ell p^{u}$ with ${\rm gcd}(\ell, p) = 1$ and $\varepsilon_{y}$ is the inseparable degree of $y$.
\item[$(vi)$] We have that
$$d\langle p^{s_{1}} e, v\rangle \geq 1 + d \langle p^{s_{1}}e, v_{y_{0}}\rangle$$
for any uncolored vertex $v$ of $\D_{y_{0}}$.
\item[$(vii)$] If $C = \P^{1}_{k}$, then
$$d\langle p^{s_{1}} e, v\rangle \geq -1 - d \langle p^{s_{1}}e, v_{\deg}\rangle$$
for any vertex $v$ of $\D_{y_{\infty}}$.
\end{itemize}
\end{definition}
From a coherent family $\theta = (\tilde{\D}, e , \underline{s}, \underline{\lambda})$ as in \ref{coherent}, we define a sequence of $k$-linear operators
$$\left\{\partial_{\theta}^{(i)}: k(\zeta)[M]\rightarrow k(\zeta)[M]\right\}_{i\geq 0}, \text{ where } k(\zeta)[M] = \bigoplus_{m\in M}k(\zeta)\chi^{m} \text{ and } \zeta = (t-y_{0})^{\frac{1}{d}}.$$
It satisfies the axioms $(i), (ii), (iv)$ of an LFIHD of Section \ref{sec-intro} (but does not satisfies the axiom $(iii)$ of an LFIHD).  For all $i, r\in \ZZ_{\geq 0}$ and $m\in M$ we let 
$$\partial^{(i)}_{\theta}((t - y_{0})^{r}\xi_{m}\chi^{m}) := 
\zeta^{-dv_{z_{0}}(m+ie)}\partial^{(i)}_{\zeta, \underline{s}, \underline{\lambda}}(\zeta^{-dv_{z_{0}}(m)}(t-y_{0})^{r})\xi_{m+e}\chi^{m+e},$$
where $\xi_{m}$ is an element of $k(t)^{\star}$ such that $\div(\xi_{m}) + \sum_{y\in C}\langle m, \tilde{v}_{y}\rangle\cdot [y] = 0$ (here $\tilde{v}_{y} = v_{y}$ for all $y\in C'\setminus\{y_{0}\}$ and $\tilde{v}_{y} = 0$ otherwise) and $\partial_{\zeta, \underline{s}, \underline{\lambda}}$ is the LFIHD on $k[\zeta]$ given by the formula
$$\sum_{j = 0}^{\infty} \partial^{(j)}_{\zeta, \underline{s}, \underline{\lambda}}(\zeta)T^{j} = \zeta  + \sum_{i =0}^{r}\lambda_{i}T^{p^{s_{i}}}$$
for a variable $T$ over $k[\zeta]$.
\\

The main result of this paper is the following (see \cite[Theorem 3.22]{FZ05}, \cite[Theorem 3.28]{Lie10a} , \cite[Theorem 5.11]{LL16} for the case of perfect ground fields).
\begin{theorem}\label{main}
Let $X$ be a complexity-one affine $\TT$-variety described by a polyhedral divisor $\D$
over a regular curve $C$.  Then the presence of a horizontal $\G_{a}$-action on $X$ implies that $C = \A^{1}_{k}$ or $C = \P^{1}_{k}$, and in this case, the map $\theta \mapsto \partial_{\theta}$ induces a bijection between the set of coherent families $\theta = (\tilde{\D}, e, \underline{s}, \underline{\lambda})$ on $\D$ and the set of horizontal LFIHDs on the $k$-algebra $k[X] = A[C, \D]$.
\end{theorem}
\begin{proof}
We first show how to a horizontal $\G_{a}$-action on $X$ we may associate a coherent family.
By Lemma \ref{ext}, we have the equality 
$k(X) \cap A[C_{\bar{k}}, \D_{\bar{k}}] = k[X],$
where $\D_{\bar{k}}$ is the polyhedral divisor corresponding to the normalization
of $X_{\bar{k}}$. Therefore, having a horizontal LFIHD on $k[X]$ is equivalent 
to considering a horizontal LFIHD $\partial$ on $A[C_{\bar{k}}, \D_{\bar{k}}]$
in which the extension $\bar{\partial}$ on the function field $\bar{k}(X_{\bar{k}})$
(defining a family of $k$-linear operators and satisfying $(i),(ii), (iv)$ of
the definition of an LFIHD; see e.g. \cite[Section 3]{Voj07} for the existence of such an extension) stabilizes $k(X)$, i.e., $\bar{\partial}^{(i)}(k(X)) \subseteq k(X)$
for any $i\geq 0$. For such an LFIHD $\partial$ let us denote by $e$ its degree.

According to Corollary \ref{l-linea}, there exists a $k$-rational point $y_{\infty}\in C$
and a maximal subcone $\omega\subseteq \sigma^{\vee}$ in which the evaluation map $m\mapsto \D(m)_{|C'}$ restricted to it is linear, where $C' = C\setminus \{y_{\infty}\}$. Moreover, there exists a family $(v_{y})_{y\in C'}$ of $N_{\QQ}$ such that 
$$\D(m)_{|C'} = \sum_{y\in C'}\langle m ,v_{y}\rangle \cdot [y] \text{ for any } m\in \omega.$$
This family also satisfies Condition $(ii)$ of Definition \ref{defcolcol} of a coloring for some  $k$-rational point $y_{0}\in C'$. The set of maximal cones in which 
the evaluation map $m\mapsto \D(m)_{|C'}$ restricted to it is linear coincides with the one for the piecewise linear map $m\mapsto\min_{v\in \deg\, \D_{|C'}}\langle m , v\rangle$ (see e.g. the discussion in \cite[Definition 1.3]{Lie10a}).
Therefore, $v_{\deg}:= \sum_{y\in C'}v_{y}$ is a vertex of $\deg\, \D_{|C'}$ and the dual of
$\omega$ is generated by $\deg\, \D_{|C'} - v_{\deg}$ (compare \cite[Lemma 1.4]{AH06}).
This shows that $\tilde{\D} = (\D, (v_{y})_{y\in C'}, y_{0})$ is a coloring.

As the field $\bar{k}$ is perfect, by \cite[Theorem 5.11]{LL16} applied to the LFIHD $\partial$ on $A[C_{\bar{k}}, \D_{\bar{k}}]$, there exists a family $\theta_{\bar{k}} = (\tilde{\D}_{\bar{k}}, e , \underline{s}, \underline{\lambda})$ which verifies Conditions $(i), (ii), (iii)$ of Definition \ref{coherent} (here the coloring $\tilde{\D}_{k}$
comes from the coloring $\tilde{\D}$ above by scalar extension) and $\partial = \partial_{\theta_{\bar{k}}}$. 

Multiplying by a kernel element in $A$ if necessary, the resulting LFIHD
would stabilize $A_{\omega}$ (see Corollary \ref{l-linea} for the definition of $A_{\omega}$).
This forces to have that $\lambda_{j}\in k^{\star}$ for any $j\in \{1, \ldots r\}$ (use \cite[Theorem 5.8]{LL16} that holds over any field).  So in order to show that $\theta = (\tilde{\D}, e , \underline{s}, \underline{\lambda})$ is a coherent family, we have to check that it satisfies Conditions $(v), (vi), (vii)$ of  Definition \ref{coherent}.

Let us write $h_{y}$, when $y\in C\setminus\{y_{0}, y_{\infty}\}$ (respectively, $h_{\bar{k}, y}$, when $y\in C_{\bar{k}} \setminus\{y_{0}, y_{\infty}\}$),
for the piecewise linear map given by 
$$h_{y}(m) = \min_{v\in \D_{y}}\langle m, v - v_{y}\rangle \text{ respectively, } h_{\bar{k}, y}(m) = \min_{v\in \D_{\bar{k}, y}}\langle m, v- v_{\bar{k}, y}\rangle$$
for any $m\in\sigma^{\vee}$. Here we let $(v_{\bar{k}, y})_{y\in C_{\bar{k}}'}$ be the family of colored vertices of $\tilde{\D}_{\bar{k}}$.
Denote by $h$ the linear extension of $(h_{y_{0}})_{|\omega}$
on the whole cone $\sigma^{\vee}$ and  by $h_{y_{\infty}}$ (respectively, $h_{\bar{k}, y_{\infty}}$) the support function $$m\mapsto \min_{v\in \D_{y_{\infty}}}\langle m, v - v_{\rm deg}\rangle \text{ respectively, } m\mapsto \min_{v\in \D_{y_{\infty}}}\langle m, v - v_{\bar{k}, \rm deg}\rangle, \text{ where }v_{\bar{k}, \rm deg} = \sum_{y\in C_{\bar{k}}'}v_{\bar{k},y}.$$ Moreover, we let $h_{y_{0}}(m) = h_{\bar{k}, y_{0}}(m) = \min_{v\in \D_{y_{0}}}\langle m, v\rangle $ for any $m\in \sigma^{\vee}$.
Then, assuming that $m + p^{s_{1}}e\in \sigma^{\vee}\cap M$ with $m\in\sigma^{\vee}\cap M$, \cite[Theorem 5.11]{LL16} gives the following conditions.
\begin{itemize}
\item[(1)] If $h_{\bar{k}, y}(m +p^{s_{1}}e)\neq 0$, then $\lfloor p^{u}h_{\bar{k}, y}(m +p^{s_{1}}e)\rfloor - \lfloor p^{u}h_{\bar{k}, y}(m)\rfloor \geq 1$ whenever $y\in C_{\bar{k}}\setminus \{y_{0}, y_{\infty}\}$.
\item[(2)] If $h_{\bar{k}, y_{0}}(m + p^{s_{1}}e)\neq h(m + p^{s_{1}}e)$, then
$\lfloor d h_{\bar{k}, y_{0}}(m + p^{s_{1}}e)\rfloor - \lfloor dh_{\bar{k}, y_{0}}(m)\rfloor \geq 1 + dh(p^{s_{1}}e)$.
\item[(3)] If $ C = \P^{1}_{\bar{k}}$, then
$\lfloor d h_{\bar{k}, y_{\infty}}(m + p^{s_{1}}e)\rfloor - \lfloor dh_{\bar{k}, y_{\infty}}(m)\rfloor -\geq 1 - dh(p^{s_{1}}e)$.
\end{itemize}
With the same condition on $m$, the definition of $\Dbar$ implies that the three last conditions are respectively equivalent to:
\begin{itemize}
\item[(4)] If $h_{y}(m +p^{s_{1}}e)\neq 0$, then $\lfloor p^{u} \varepsilon _{y}\cdot h_{y}(m +p^{s_{1}}e)\rfloor - \lfloor p^{u} \varepsilon _{y}\cdot h_{y}(m)\rfloor \geq 1$ whenever $y\in C\setminus \{y_{0}, y_{\infty}\}$.
\item[(5)] If $h_{y_{0}}(m + p^{s_{1}}e)\neq h(m + p^{s_{1}}e)$, then
$\lfloor d h_{y_{0}}(m + p^{s_{1}}e)\rfloor - \lfloor dh_{y_{0}}(m)\rfloor \geq 1 + dh(p^{s_{1}}e)$.
\item[(6)] If $ C = \P^{1}_{k}$, then
$\lfloor d h_{y_{\infty}}(m + p^{s_{1}}e)\rfloor - \lfloor dh_{y_{\infty}}(m)\rfloor -\geq 1 - dh(p^{s_{1}}e)$.
\end{itemize}
Therefore for showing that $\theta = (\tilde{\D}, e , \underline{s}, \underline{\lambda})$ is a coherent family, it suffices to prove that Conditions $(v), (vi), (vii)$ of  Definition \ref{coherent} are respectively equivalent to Conditions (4), (5), (6). Let us prove the equivalence $(4) \Leftrightarrow (v)$. The others equivalences, namely $(5)\Leftrightarrow (vi)$ and $(6)\Leftrightarrow (vii)$, are shown in the same way and
are left to the reader. Assume that $(4)$ holds. Take $y\in C'$ and
let $\omega_{0}\neq \omega$ be a maximal cone in which $m\in \omega_{0}\mapsto h_{y}(m)$
is linear (note that if such an $\omega_{0}$ does not exist, then  $h_{y}$ is identically zero and
$(4) \Leftrightarrow (v)$ is true). We also denote by $h_{y, \omega_{0}}$ the linear extension on $\sigma^{\vee}$ of $(h_{y})_{|\omega_{0}}$. Then for $m\in\omega_{0}\cap M$ such that $h_{y}(m)\in \mathbb{Z}$ and $m + p^{s_{1}}e\in \omega_{0}$  we
have that
$$\lfloor p^{u} \varepsilon _{y}\cdot h_{y}(m + p^{s_{1}}e)\rfloor = p^{u} \varepsilon _{y}\cdot h_{y}(m) + \lfloor p^{u} \varepsilon _{y}\cdot h_{y, \omega_{0}}(p^{s_{1}}e)\rfloor.$$
Therefore using $(4)$, we get that $p^{u} \varepsilon _{y}\cdot h_{y, \omega_{0}}(p^{s_{1}}e) \geq \lfloor p^{u} \varepsilon _{y}\cdot h_{y, \omega_{0}}(p^{s_{1}}e)\rfloor\geq 1$. Now remarking that such maximal cones $\omega_{0}\neq \omega$ as previously are in bijection with vertices $v$ of $\D_{y}$ 
different from $v_{y}$ via $h_{y, \omega_{0}}(m) = \langle m, v-v_{y}\rangle$,
we conclude that $(v)$ is satisfied. 

Let us now assume $(v)$. Let $m\in \sigma^{\vee}\cap M$ such that $m + p^{s_{1}}e\in \sigma^{\vee}\cap N$ and with $h_{y}(m + p^{s_{1}}e)\neq 0$. Then there exists a maximal 
cone $\omega_{0}\neq \omega$ where $(h_{y})_{|\omega_{0}}$ is linear such that $m + p^{s_{1}}e\in \omega_{0}$. Thus,
$$ \lfloor p^{u} \varepsilon _{y}\cdot h_{y}(m +p^{s_{1}}e)\rfloor \geq  \lfloor p^{u} \varepsilon _{y}\cdot h_{y}(m)\rfloor + \lfloor p^{u} \varepsilon _{y}\cdot h_{y, \omega_{0}}(p^{s_{1}}e)\rfloor \geq \lfloor p^{u} \varepsilon _{y}\cdot h_{y}(m)\rfloor  +1,$$
where the last inequality comes from $(v)$. This establishes $(4) \Leftrightarrow (v)$.

Our analysis implies that we get an injective map from the set of horizontal LFIHDs on $k[X]$ to the set of coherent families of $\D$ (the verification of the injectivity being formal). It remains to check that for a given  coherent family $\theta$ on $\D$, there is a horizontal $\G_{a}$-action on $X$ corresponding to it. The sequence of operators $\partial_{\theta}$ of $k(X)$ extends to one on $\bar{k}(X_{\bar{k}})$. As the previous conditions (1), (2), (3) are satisfied, by \cite[Theorem 5.11]{LL16}, it defines a horizontal LFIHD $\partial$ on $A[C_{\bar{k}}, \D_{k}]$. We conclude using  Lemma \ref{ext} that $\partial$ stabilizes $k[X]$. This gives the required $\G_{a}$-action and finishes the proof of the theorem.
\end{proof}
\subsection{Some examples}
We start by bringing an example that involves the inseparable degree of  Definition 
\ref{coherent} $(v)$.
\begin{example}\label{final-ex}
Here the base field $k$ is of characteristic $2$. We assume that $N_{\QQ} = \QQ$
and that $\sigma = \{0\}$. We consider the $\sigma$-polyhedral divisor $\D$ over $\A^{1}_{k} = \spec\, k[t]$ supported in two points, and given by the formula
$$\D = \left\{\frac{1}{5}\right\}\cdot [0] + \left[0, \frac{1}{5}\right]\cdot [y],$$
with $y\neq 0$. Note that for the coloring $\tilde{\D} = (\D, (v_{0} = \frac{1}{5}, v_{y'} = 0 \text{ for } y'\in \A^{1}_{k}\setminus\{0\}), 0)$, we have that
$$\tilde{\tau} = \QQ_{\geq 0}(1, 0) + \QQ_{\geq 0}(1,5).$$ 
If we take $e =1$ and $s_{1} =2$, then $(2^{s_{1}}e, -\frac{1}{5}-\frac{2^{s_{1}}e}{5}) = (4, -1)$ is a
Demazure root of $\tilde{\tau}$ with distinguished ray $\QQ_{\geq 0}\cdot (1, 5)$. It was observed in \cite[Example 5.13]{LL16} that if $k$ is algebraically closed, then $X = \spec\, A[\A^{1}_{k}, \D]$ 
is isomorphic to the affine surface
$$W_{2, 5} = \{( x, w, z)\in \A_{k}^{3}\,|\, x^{2}w = x + z^{5}\}$$
and that the family $\theta = (\tilde{\D}, e, (2), (1))$ is not coherent (as Condition $(v)$ of Definition \ref{coherent} fails to be satisfied). 

We claim that if $k$ is imperfect, then $\theta$ defines a horizontal $\G_{a}$-action for a well-chosen $y$. Indeed, let us pick a non-square element $\lambda$ in $k$ and assume that $y$ is given by the polynomial $p_{y}(t) = t^{2}-\lambda$. Then $\varepsilon_{y} =2$ and $\theta$ is coherent (as now Condition $(v)$ of Definition \ref{coherent} is satisfied). One can explicitly check this on the graded pieces of  $A[\A^{1}_{k}, \D]$. Let us denote $$A_{m} = H^{0}(\mathbb{A}^{1}_{k}, \mathcal{O}_{\A^{1}_{k}}(\lfloor \D(m)\rfloor))\chi^{m} \text{ so that }k[X] = \bigoplus_{m\in\ZZ} A_{m}.$$ If $t^{r}\chi^{m}\in A_{\geq 0} := \bigoplus_{m\geq 0} A_{m}$ is a homogeneous element, then $-5r\leq m$, 
$$\partial_{\theta}^{(4j)}(t^{r}\chi^{m}) = \binom{5r + m}{j}t^{r-j}\chi^{m + 4j}\in A_{\geq 0}\text{ and } \partial_{\theta}^{(i)}(t^{r}\chi^{m}) = 0 \text { for }i\not\in 4\ZZ_{\geq 0}.$$
Moreover, $\partial_{\theta}^{(j)}((t^{2}-\lambda)t\chi^{-5})\in A_{\geq 0}$ for any $j\in \ZZ_{\geq 0}$. Since $A_{\geq 0}\cup\{(t^{2}-\lambda)t\chi^{-5}\}$ generates the $k$-algebra $k[X]$, we conclude that $\partial_{\theta}$ gives rise to a horizontal $\G_{a}$-action on $X$. Finally, remark that $X$ is not geometrically normal (as $\spec \, A_{\geq 0}$ is isomorphic to $\spec\, k[x_{1}, x_{2}, x_{3}]/((x_{1}^{2} - \lambda)x_{3} - x_{2}^{2})$).
\end{example}

Beside the inseparability condition on each point involving in Definition \ref{coherent}, one may ask why we need the factor $p^{u}$ in Condition $(v)$. Originally, it appears in the proof of \cite[Lemma 5.10]{LL16} because we lift a $\G_{a}$-action
from a cyclic covering (whose the degree might be divisible by the characteristic of the base field). Therefore, the factor $p^{u}$ comes from a ramification phenomenon. The next example aims to illustrate this technical point. 
\begin{example}
We assume that the base field $k$ is algebraically closed. 
We consider the polyhedral divisor $\D$ over the affine line $\A^{1}_{k} = \spec\, k[t]$ defined 
by 
$$\D_{0} = \left(\frac{1}{2}, 0\right) + \sigma,\, \D_{1}= \left[\left(\frac{1}{2}, 0\right), \left(0, 1\right)\right]+\sigma\text{  and }\D_{y} = \sigma \text{ for all }y\in \A^{1}_{k},$$
where the lattice $N$ is $\ZZ^{2}$ and the tail cone $\sigma$ is $\QQ_{\geq 0}^{2}$.
Note that, regarding the notations of Definitions \ref{defcolcol} and \ref{coherent}, we have  $d = 2$, $C' = \A^{1}_{k}$, and 
$$\deg\, \D_{|C'} - v_{\rm deg} = \left[\left(\frac{1}{2}, -1\right), \left(0, 0\right)\right] + \sigma,$$
where we take the coloring $(v_{0} =(\frac{1}{2}, 0), v_{1} = (0, 1))$. This implies that
$$\omega = \QQ_{\geq 0}(0,1) + \QQ_{\geq 0}(2, 1)\subseteq M_{\QQ}\text{ and } \tilde{\tau} = \QQ_{\geq 0} (1, -2, 0) + \QQ_{\geq 0}(0, 1, 0) + \QQ_{\geq 0}(1, 0, 2).$$
Now for defining a coherent family we choose a Demazure root $\tilde{e}$ of the cone $\tilde{\tau}$ with distinguished ray $\QQ_{\geq 0}(1, 0, 2)$. For example, we take $\tilde{e} = (1, 0, -1)$ and set $e = (1, 0)$. One sees that if the
characteristic of $k$ is unequal to $2$, then Condition $(v)$ of Definition \ref{coherent} is not satisfied. However, if the characteristic is equal to $2$, due to the ramification phenomenon (presence of the factor $p^{u}$),
Condition $(v)$ of Definition \ref{coherent} is indeed satisfied. So this means that, from the combinatorial data discussed before, the $\G_{a}$-action only shows up in characteristic $2$. Let us check this in example. 

The graded pieces $A_{m_{1}, m_{2}}$ of the algebra $A = A[\A^{1}_{k}, \D]$ can be cut into two regions according to the cones 
$$ \omega = \omega_{1} = \QQ_{\geq 0}(1, 0) + \QQ_{\geq 0}(2,1) \text{ and }\omega_{2} = \QQ_{\geq 0}(0, 1) + \QQ_{\geq 0}(2,1).$$
In those cones the evaluation map $(m_{1}, m_{2})\mapsto\D(m_{1}, m_{2})$ is linear, that is,

$$A_{m_{1},m_{2}} = k[t]t^{-\lfloor \frac{m_{1}}{2}\rfloor}(t-1)^{-m_{2}}\chi^{(m_{1}, m_{2})} \text{ if }(m_{1}, m_{2})\in \omega_{1}, \text{ and }$$
$$A_{m_{1},m_{2}} = k[t]t^{-\lfloor \frac{m_{1}}{2}\rfloor}(t-1)^{-\lfloor\frac{m_{1}}{2}\rfloor}\chi^{(m_{1}, m_{2})} \text{ if }(m_{1}, m_{2})\in \omega_{2}.$$ 
From this, we note that there is an isomorphism
of graded rings
$$\psi:A_{\omega} = \bigoplus_{(m_{1}, m_{2})\in \omega\cap \ZZ^{2}}A_{m_{1}, m_{2}}\rightarrow k[\tilde{\omega}\cap \ZZ^{3}] = \bigoplus_{(r, m_{1}, m_{3})\in \tilde{\omega}\cap \ZZ^{3}}kt^{r}\chi^{(m_{1}, m_{2})},$$
$$t\mapsto t,\,\, \xi_{m_{1}, m_{2}}\chi^{(m_{1}, m_{2})}\mapsto \chi^{(m_{1}, m_{2})},$$
where $\xi_{m_{1}, m_{2}} =(t-1)^{-m_{2}}$.
Lifting the homogeneous LFIHD of  $k[\tilde{\omega}\cap \ZZ^{3}]$ attached to the Demazure root $\tilde{e}$ via $\psi$ gives the LFIHD $\partial$ on $A_{\omega}$ defined as
$$\partial^{(i)}(t^{r} \xi_{m_{1}, m_{2}}\chi^{(m_{1}, m_{2})}) = \binom{2r + m_{1}}{i}t^{r-i}\xi_{m_{1} + i, m_{2}}\chi^{(m_{1} + i, m_{2})} \text{ for } i= 0, 1,2, \ldots $$
Condition $(v)$ in Definition \ref{coherent} precisely means that $\partial$ extends to an LFIHD on the whole algebra $A$. Denote by the same letter $\partial$ the extension
of $\partial$ in the fraction field of $A$.
When the characteristic
is unequal to $2$, we see that  $\partial^{(1)}(A_{0,1})\not\subseteq A$ since
$$\partial^{(1)}(\chi^{(0,1)}) = \partial^{(1)}(t\xi_{0, 1}\chi^{(0,1)}- \xi_{0, 1}\chi^{(0,1)}) = 2t^{-1}\chi^{(1,1)}\not\in A.$$
Now suppose that the characteristic is equal to $2$ and observe that
$$A_{\omega_{2}} =  \bigoplus_{(m_{1}, m_{2})\in \omega_{2}\cap \ZZ^{2}}A_{m_{1}, m_{2}} = k[t, \chi^{(0,1)}, \chi^{(1,1)}, z :=t^{-1}(t-1)^{-1}\chi^{(2, 1)}].$$
We obviously have $\partial^{(i)}(t)\in A$ for any $i \in \ZZ_{\geq 0}$ and $z\in \ker(\partial)$. From a direct computation, we get 
$$\partial^{(0)}(\chi^{(0,1)}) = \chi^{(0,1)},\partial^{(1)}(\chi^{(0,1)}) = 0, \partial^{(2)}(\chi^{(0,1)}) =  z, \partial^{(j)}(\chi^{(0,1)}) = 0 \text{ for all }i\geq 3, \text{ and }$$
$$\partial^{(0)}(\chi^{(1,1)}) = \chi^{(1,1)}, \partial^{(1)}(\chi^{(1,1)}) = tz, \partial^{(2)}(\chi^{(1,1)}) = t^{-1}\chi^{(3,1)}\in A,$$ 
$$\partial^{(3)}(\chi^{(1,1)}) = t^{-2}(t-1)^{-1}\chi^{(4,1)}\in A,
\partial^{(j)}(\chi^{(1,1)}) = 0\text{ for all }j\geq 4.$$ 
That shows that, in characteristic $2$, the sequence $\partial$ is an LFIHD on $A$.

\end{example}

\emph{Acknowledgment.} We would like to thank the anonymous referee for her (or his) valuable comments which
allowed to improve the presentation.
We also warmly thank David Bradley-Williams for some corrections.
This research was conducted in the framework of the research training group
\emph{GRK 2240: Algebro-geometric Methods in Algebra, Arithmetic and Topology},
which is funded by the DFG.

\end{document}